\documentclass[12pt]{amsart}

\usepackage{amssymb}
\usepackage{enumerate}
\usepackage{subfigure}
\usepackage{setspace}
\usepackage{graphicx}

\numberwithin{equation}{section}

\newtheorem{thm}{Theorem}[section]
\newtheorem{prop}[thm]{Proposition}
\newtheorem{cor}[thm]{Corollary}
\newtheorem{lem}[thm]{Lemma}
\newtheorem{conj}[thm]{Conjecture}
\newtheorem{obs}[thm]{Observation}

\theoremstyle{definition}

\newcounter{countcases}

\newcommand{\PBS}[1]{\let\temp=\\#1\let\\=\temp}  

\numberwithin{figure}{section}

\begin{document}
\author{Brendon Rhoades}
\email{brhoades@math.mit.edu}
\address{Brendon Rhoades, Department of Mathematics, Massachusetts Institute of Technology, Cambridge, MA, 02139}
\title[Cyclic sieving and cluster multicomplexes]
{Cyclic sieving and cluster multicomplexes}


\bibliographystyle{../elsart-num-sort}

\date{\today}

\begin{abstract}
Reiner, Stanton, and White \cite{RSWCSP} proved results regarding the enumeration of 
polygon dissections up to rotational symmetry.  Eu and Fu \cite{EuFu} generalized these
results to Cartan-Killing types other than A by means of actions of 
deformed
Coxeter elements
on cluster complexes of Fomin and Zelevinsky \cite{FZY}.  The Reiner-Stanton-White and
Eu-Fu results were proven using direct counting arguments.  We give representation 
theoretic proofs of closely related results using the notion of noncrossing and 
seminoncrossing tableaux due to Pylyavskyy \cite{PN} as well as some geometric
realizations of finite type cluster algebras 
due to Fomin and Zelevinsky
\cite{FZClusterII}. 
\end{abstract}
\maketitle

\section{Introduction and Background}

Let $X$ be a finite set and let $C = \langle c \rangle$ be a finite cyclic group acting on $X$ with 
distinguished generator $c$.  Let $X(q) \in \mathbb{N}[q]$ be a polynomial in $q$ with nonnegative 
integer coefficients and let $\zeta \in \mathbb{C}$ be a root of unity with the same multiplicative 
order as $c$.  Following Reiner, Stanton, and White \cite{RSWCSP} we say that the triple $(X, C, X
(q))$ \emph{exhibits the cyclic sieving phenomenon (CSP)} if for all integers $d \geq 0$, the 
cardinality of the fixed point set $X^{c^d}$ is equal to the polynomial evaluation $X(\zeta^d)$.  
Given a finite set $X$ equipped with the action of a finite cyclic group $C = \langle c \rangle$, since 
the cycle type of the image of $c$ under the canonical homomorphism $C \rightarrow \mathfrak
{S}_X$ is determined by the fixed point set sizes $|X^{c^d}|$ for $d \geq 0$, finding a polynomial $X
(q)$ such that $(X, C, X(q))$ exhibits the CSP completely determines the enumerative structure of
 the action of $C$ on $X$.  It can be shown that given a finite set $X$ 
with the action of a finite cyclic group $C = \langle c \rangle$, it is always possible to find a
polynomial $X(q)$ such that $(X, C, X(q))$ exhibits the CSP:  for example, if the order of $C$
is $n$, we can take $X(q) = \sum_{i = 0}^{n-1} a_i q^i$, where $a_i$ is equal to the number of $C$-orbits
in $X$ whose stabilizer order divides $i$ \cite[Definition-Proposition, p. 1]{RSWCSP}.  
The interest in proving a CSP lies in finding 
a `nice' formula for $X(q)$, ideally making no explicit reference to the action of $C$ on $X$.  
These `nice' formulas for $X(q)$ are typically either simple sums or products of $q$-analogues of numbers
or binomial coefficients or are generating functions for some natural statistic $stat: X \rightarrow \mathbb{N}$, i.e.,
$X(q) = \sum_{x \in X} q^{stat(x)}$.  In this paper, the set $X$ will typically consist of objects related to noncrossing dissections of a regular $n$-gon and the group $C$ will act by 
an appropriate version of rotation.

Let $\Phi$ be a root system and let $\Pi \subset \Phi$ be a system of 
simple roots within $\Phi$.  The choice of the simple system $\Pi$ partitions
$\Phi$ into two subsets $\Phi = \Phi_{>0} \uplus \Phi_{< 0}$ of 
\emph{positive} and \emph{negative} roots.  Let 
$\Phi_{\geq -1} := \Phi_{> 0} \uplus - \Pi$ denote the set of
\emph{almost positive} roots, i.e., roots in $\Phi$ which are either positive
or negatives of simple roots.  In 2003 Fomin and Zelevinsky 
\cite{FZY}
introduced
a simplicial complex $\Delta(\Phi)$ called the \emph{cluster complex}
whose ground set is $\Phi_{\geq -1}$.  A certain 
deformed Coxeter element $\tau$ 
in the Weyl group of $\Phi$ 
arising from a bipartition of the associated Dynkin diagram
acts on the set $\Phi_{\geq -1}$ of almost
positive roots and induces a (simplicial) action on the complex
$\Delta(\Phi)$ which preserves dimension.  When $\Phi$ is of type ABCD,
the action of $\tau$ on the set of faces of $\Delta(\Phi)$ of a fixed
dimension is isomorphic to the action of the rotation operator
on a certain set of noncrossing polygon dissections with a fixed number
of edges (where in type D the definitions of `rotation' and `noncrossing'
differ slightly from those for classical polygon dissections).  We outline
the corresponding actions on polygon dissections in types ABCD.  For
$n \geq 3$, we denote by $\mathbb{P}_n$ the regular $n$-gon.

For $n \geq 3$,
let $\Phi_{A_{n-3}}$ be a root system of type A$_{n-3}$.  
For $k \geq 0$ fixed,
the action of
the Coxeter element $\tau$ on the set 
of $(k-1)$-dimensional faces of the cluster complex
$\Delta(\Phi_{A_{n-3}})$ 
is isomorphic to the action of rotation on the set of noncrossing
dissections of $\mathbb{P}_n$ with exactly $k$ diagonals. 

It turns out that the action of the Coxeter element $\tau$ on the
cluster complex $\Delta(\Phi)$ is the same for $\Phi$ of type B or C.
For reasons related to geometric realizations of the types B and C
cluster algebras \cite{FZClusterII} we will use the descriptor
`C' in this paper.
For $n \geq 2$, let $\Phi_{C_{n-1}}$ be a root system of type C$_{n-1}$.
For $k \geq 0$, the action of $\tau$ on the $(k-1)$-dimensional faces
of $\Delta(\Phi_{C_{n-1}})$ is isomorphic to the action of rotation on
the set of centrally symmetric dissections of $\mathbb{P}_{2n}$ with
exactly
$k$ diagonals, where a pair of centrally symmetric nondiameter diagonals 
of $\mathbb{P}_{2n}$ counts as a \emph{single} diagonal.

Let $\Phi_{D_n}$ denote a root system of type D$_n$ for $n \geq 2$.  To 
realize the action of the Coxeter element $\tau$ on the cluster complex
$\Delta(\Phi_{D_n})$ as an action on dissection-like objects, we must
slightly modify our definitions of noncrossing dissections and rotation.
A \emph{D-diagonal} in $\mathbb{P}_{2n}$ is either a pair of 
centrally symmetric nondiameter diagonals or a diameter colored one
of two colors, solid/blue or dotted/red.  Two D-diagonals are said to \emph{cross}
if they cross in the classical sense, except that distinct diameters
of the same color do not cross and identical diameters of different 
colors do not cross.  A \emph{D-dissection} of $\mathbb{P}_{2n}$
is a collection of pairwise noncrossing D-diagonals in $\mathbb{P}_{2n}$.
\emph{D-rotation} acts on D-dissections by classical rotation, except that
D-rotation switches the color of diameters.  In particular, the action of
D-rotation on D-dissections of $\mathbb{P}_{2n}$ has order 
$n$ if $n$ is even and order $2n$ if $n$ is odd.  The action of the 
Coxeter element $\tau$ on the set of $(k-1)$-dimensional faces of the 
cluster complex $\Delta(\Phi_{D_n})$ is isomorphic to the action
of D-rotation on the set of D-dissections of $\mathbb{P}_{2n}$ with 
exactly $k$ D-diagonals (where a pair of uncolored centrally 
symmetric nondiameters counts as a \emph{single} D-diagonal).  

The following CSPs involving the actions of deformed Coxeter elements 
on cluster complexes were proven by Reiner, Stanton, and White 
\cite{RSWCSP} in the case of type A and by Eu and Fu \cite{EuFu}
in the cases of types B/C and type D.  We use the standard $q$-analog notation
\begin{align*}
[m]_q &:= 1 + q + \cdots + q^{m-1}, \\
[m]!_q &:= [m]_q [m-1]_q \cdots [2]_q [1]_q, \\
{m \brack r}_q &:= \frac{ [m]!_q } { [r]!_q [m-r]!_q },
\end{align*}
for $m \geq r > 0$.

\begin{thm}  Fix $k \geq 0$. \\
1. \cite[Theorem 7.1]{RSWCSP}
For $n \geq 3$
let $X$ be the set of noncrossing
dissections of $\mathbb{P}_n$ with exactly $k$  diagonals.  Let $C = \mathbb{Z}_n$ act on $X$ by rotation.  The triple 
$(X, C, X(q))$ exhibits the cyclic sieving phenomenon, where
\begin{align*}
X(q) &= \frac{1}{[n+k]_q} {n+k \brack k+1}_q {n-3 \brack k}_q \\
  &= \frac{[n+k-1]!_q}{[k]!_q [k+1]!_q [n-k-3]!_q [n-1]_q [n-2]_q}.
\end{align*} \\
2. \cite[Theorem 4.1, $s = 1$]{EuFu}
For $n \geq 2$
let $X$ be the set of centrally symmetric dissections of
$\mathbb{P}_{2n}$ with 
exactly
$k$ noncrossing diagonals, where a pair of centrally symmetric nondiameter diagonals counts
as a \emph{single} diagonal.  
Let the cyclic group $C = \mathbb{Z}_n$ of order $n$ act on $X$ by rotation.  The triple 
$(X, C, X(q))$ exhibits the cyclic sieving phenomenon, where
\begin{equation*}
X(q) = {n+k+1 \brack k}_{q^2} {n+1 \brack k}_{q^2}.
\end{equation*} \\
3. \cite[Theorem 5.1, $s = 1$]{EuFu}
For $n \geq 2$
let $X$ be the set of D-dissections of $\mathbb{P}_{2n}$ with exactly $k$ D-diagonals.  Let the cyclic group $C = \mathbb{Z}_{2n}$ 
of order $2n$
act on $X$ by D-rotation.  The triple $(X, C, X(q))$ exhibits the cyclic sieving phenomenon, where
\begin{align*}
X(q) = & {n+k-1 \brack k}_{q^2} {n-1 \brack k}_{q^2} + {n + k - 1 \brack k}_{q^2} {n-2 \brack k-1}_{q^2} \cdot q^n \\
   & + {n+k-1 \brack k}_{q^2} {n-2 \brack k-2}_{q^2} + 
     {n+k-2 \brack k}_{q^2}{n-2 \brack k-2}_{q^2} \cdot q^n.
\end{align*}
\end{thm}

Specializing Part 1 of Theorem 1.1 to $q = 1$, we get that 
the number of dissections of $\mathbb{P}_n$ with $k$ noncrossing diagonals is
\begin{equation*}
\frac{1}{n+k} {n+k \choose k+1} {n-3 \choose k} = \frac{(n+k-1)!}{k! (k+1)! (n-k-3)! (n-1)(n-2)}.
\end{equation*}
This enumeration was proven first by Cayley \cite{Cayley}.
O'Hara and Zelevinsky noted that the Frame-Robinson-Thrall hook length formula \cite{FRTHook} implies that this expression is also equal to the number of standard Young tableaux of shape 
$(k+1)^2 1^{n-k-3}$.  
When $k = n-3$, the dissections
in question are in fact triangulations 
and the above expression specializes to the Catalan number 
$C_{n-3} = \frac{1}{n-2} {2n - 6 \choose n-3}$.  Eu and Fu proved Parts 1-3
of Theorem 1.1 in the more general context of \emph{s-divisible} polygon
dissections, and hence in the context of the actions of deformed Coxeter
elements on the generalized cluster complexes of Fomin and Reading \cite{FR}.  The main purpose of this paper is to use representation theoretic
methods motivated by the theory of cluster algebras to prove CSPs
which are related to the CSPs in Theorem 1.1. 

More precisely,
given a finite set $X$ acted on by a finite cyclic group $C = \langle c \rangle$
and a polynomial $X(q) \in \mathbb{N}[q]$, there are 
essentially two main methods that have been used to show that the triple
$(X, C, X(q))$ exhibits the CSP.  
On its face, the statement that $(X, C, X(q))$ exhibits the CSP is purely enumerative.
A direct enumerative proof of such a CSP consists of 
counting the fixed point sets $X^{c^d}$ for all $d \geq 0$ and 
showing that these numbers are equal to the polynomial $X(q)$ specialized at appropriate
roots of unity.  This is how Theorem 1.1 was proven in \cite{RSWCSP} and \cite{EuFu}.  
A more algebraic approach dating back to Stembridge \cite{StemTab} in the context of the $q = -1$ 
phenomenon is as follows.  
Suppose we have a $\mathbb{C}$-vector space $V$ with distinguished basis
$\{ e_x \,|\, x \in X \}$ indexed by elements of $X$.  Suppose further that $V$ is acted on by a group
$G$ and that an element $g \in G$ satisfies
\begin{equation*}
g.e_x = e_{c.x}
\end{equation*}
for all $x \in X$.  Then, for any $d \geq 0$, the fixed point set cardinality 
$|X^{c^d}|$ is equal to the character evaluation $\chi(g^d)$, where $\chi: G \rightarrow \mathbb{C}$ 
is the character of $V$.  
It is frequently the case that 
representation theoretic properties of $V$ and/or group theoretic properties 
of $G$ can be used to equate the character
evaluation $\chi(g^d)$ with the specialization of the polynomial $X(q)$ at an appropriate root 
of unity.  
For example, in \cite{RCSP} this method is used to prove that $(X, C, X(q))$ exhibits the CSP, where
$X$ is the set of standard Young tableau of fixed rectangular shape $\lambda \vdash n$, the cyclic
group
$C = \mathbb{Z}_n$ acts on $X$ by jeu-de-taquin promotion, and $X(q) = f^{\lambda}(q)$ is a
 $q$-shift of
the generating function for major index on standard tableaux of shape $\lambda$.  In the 
proof of this result, the module used is the irreducible 
$\mathfrak{S}_n$-module of shape $\lambda$ taken with respect to its Kazhdan-Lusztig 
cellular basis and the group element which models the action of $C$ on $X$ is the 
long cycle $(1,2,\dots,n) \in \mathfrak{S}_n$. 

Since the definition of the CSP is entirely combinatorial, it is appealing to have a direct
enumerative proof of a CSP.
However, in some cases such as the action of jeu-de-taquin promotion on
rectangular standard tableaux above only a
representation theoretic proof is known.  
Moreover, many enumerative proofs of cyclic sieving phenomena involve tricky counting
arguments and/or polynomial evaluations.  Representation theoretic proofs of CSPs can be
more elegant than their enumerative counterparts, as well as give algebraic insight into 
`why' the CSP holds.  It is the purpose of this paper to prove by representation theoretic means
a pair of CSPs $(X, C, X(q))$ 
in Theorems 2.5 and 3.4
involving actions which are closely related to the actions in 
Parts 1 and 2 of Theorem 1.1 (roughly speaking, our sets $X$ will be obtained by allowing edges to occur with
multiplicity and allowing boundary edges to be omitted and our groups $C$ will act by rotation).  
We will also prove by 
a hybrid of algebraic and
enumerative means Theorem 4.6 which is a `multiplicity
counting' version of Part 3 of Theorem 1.1.  One feature of our CSPs in
Theorems 2.5, 3.4, and 4.6 
is that the polynomials $X(q)$ involved will be more representation 
theoretically suggestive than the polynomials appearing in the CSPs
of Theorem 1.1.  In particular, our polynomials $X(q)$ will be (at 
least up to $q$-shift)
the principal specializations of certain symmetric functions arising
as Weyl characters of the modules involved in our proofs.

The representation theory involved in our proofs of Theorems 2.5 and 3.4 is
motivated by the theory of finite type cluster algebras.  Cluster algebras
are a certain class of commutative rings introduced by Fomin and
Zelevinsky \cite{FZClusterI}.  Every cluster algebra comes equipped
with a distinguished generating set of \emph{cluster variables}
which are grouped into finite overlapping sets called \emph{clusters}, all
of which have the same cardinality.  (The common size of these clusters
is called the \emph{rank} of the cluster algebra.)  The cluster algebras
having only finitely many clusters enjoy a classification analogous to
the Cartan-Killing classification of finite real reflection groups
\cite{FZClusterII}.  These cluster algebras are said to be of 
\emph{finite type}.  Any finite type cluster algebra has a linear basis
consisting of \emph{cluster monomials}, i.e., monomials in the cluster
variables drawn from a fixed cluster together with a set of \emph{frozen}
or \emph{coefficient variables} which only depends on the cluster
algebra in question.  

It turns out that the cluster algebras of types ABCD are `naturally
occurring'.  
More precisely, 
in \cite{FZClusterII} Fomin and Zelevinsky endow rings related to
the coordinate ring of the Grassmannian $Gr(2,n)$ of 2-dimensional
subspaces of $\mathbb{C}^n$ with the structure of a cluster algebra 
of types A, B, and D.  Cluster algebras of type C are given a similar
geometric realization in \cite{FZClusterII}.  As finite type cluster
algebras, these rings inherit linear bases of cluster monomials.  
In Sections 2 and 3 we use the geometric realizations of the types A
and C cluster algebras presented in \cite{FZClusterII} to give
representation theoretic proofs of multiplicity counting versions of
Parts 1 and 2 of Theorem 1.1.  In Section 4 we will use the geometric
realization of the type D cluster algebra in \cite{FZClusterII} 
together with some combinatorial reasoning to 
prove a multiplicity counting version of Part 3 of Theorem 1.1.

For the rest of the paper we will use the following notation related
to symmetric functions, following the conventions of
\cite{StanEC2} and \cite{Sag}.  
For $n \geq 0$
a \emph{partition}
$\lambda$ of $n$ is a weakly decreasing sequence of positive integers
$\lambda = (\lambda_1 \geq \lambda_2 \geq \cdots \geq \lambda_{k} > 0)$
such that $\lambda_1 + \cdots + \lambda_{k} = n$.  
The number $k$ is the \emph{length} $\ell(\lambda)$ 
of $\lambda$
and $\lambda$ is said to
have $k$ \emph{parts}.
We write
$\lambda \vdash n$ to mean that $\lambda$ is a partition of $n$.  For
example, we have $(4,2,2) \vdash 8$ and
$\ell((4,2,2)) = 3$.  The \emph{Ferrers diagram} of
a partition $\lambda = (\lambda_1, \dots, \lambda_{k})$ is the 
figure consisting of $k$ left-justified rows of dots with $\lambda_i$
dots in row $i$ for $1 \leq i \leq k$.  

A \emph{$\lambda$-tableau} 
$T$ is an assignment of a positive integer to each dot in the Ferrers
diagram of $\lambda_i$.  
The partition $\lambda$ is the \emph{shape} of the 
$\lambda$-tableau $T$.
A tableau $T$ is called 
\emph{semistandard} if the entries in $T$ increase weakly across rows
and increase strictly down columns.  
A semistandard tableau $T$ of shape $\lambda \vdash n$ 
is called \emph{standard} if each of the letters $1, 2, \dots, n$ occur
exactly once in $T$.
The \emph{content} $\mathrm{cont}(T)$
of a $\lambda$-tableau $T$ is the sequence 
$\mathrm{cont}(T) = (\mathrm{cont}(T)_1, \mathrm{cont}(T)_2, \dots)$,
where 
$\mathrm{cont}(T)_i$ 
is the number of $i's$ in $T$ for all $i$.  Given a 
partition $\lambda$, the \emph{Schur function} 
$s_{\lambda}(x_1, x_2, \dots, x_n)$ 
in $n$ variables
is the polynomial in the variable set $\{ x_1, \dots, x_n \}$ defined by
\begin{equation*}
s_{\lambda}(x_1, \dots, x_n) = \sum_T x_1^{\mathrm{cont}(T)_1} \cdots
x_n^{\mathrm{cont}(T)_n},
\end{equation*}
where the sum ranges over all semistandard tableaux $T$ of shape
$\lambda$ and entries bounded above by $n$.  For 
$k \geq 0$, the \emph{homogeneous symmetric function}
$h_k(x_1, \dots, x_n)$ is given by
$h_k(x_1, \dots, x_n) := \sum_{1 \leq i_1 \leq i_2 \leq \dots \leq i_k \leq n}
x_{i_1} \cdots x_{i_k}$.  We have that 
$h_k(x_1,\dots,x_n) = s_{(k)}(x_1,\dots,x_n)$.  Given a partition 
$\lambda = (\lambda_1, \dots, \lambda_m)$, we 
extend the definition of the homogeneous symmetric functions by
defining
$h_{\lambda}(x_1,\dots,x_n)$ to be the product 
$h_{\lambda_1}(x_1, \dots,x_n) \cdots
h_{\lambda_m}(x_1, \dots, x_n)$.  Given a partition $\lambda$ and
$k \geq 0$, Pieri's Rule states that the product
$h_k(x_1,\dots,x_n) s_{\lambda}(x_1,\dots,x_n)$ is equal to 
$\sum_{\mu} s_{\mu}(x_1,\dots,x_n)$, where $\mu$ ranges over the set of
all partitions obtained by adding $k$ dots to the Ferrers diagram of 
$\lambda$ such that no two dots are added in the same column.

\section{Type A}

Our analog of Part 1 of Theorem 1.1 will involve an action on 
\emph{A-multidissections} of $\mathbb{P}_n$ which are, roughly speaking, 
noncrossing
dissections of $\mathbb{P}_n$ where boundary edges may be omitted and
edges can occur with multiplicity.  More formally, if 
$n > 2$ and
$E_A = { [n] \choose 2 }$
is the set of edges in $\mathbb{P}_n$, an A-multidissection is
a function $f : E_A \rightarrow \mathbb{N}$ such that
whenever $e, e' \in E_A$ are crossing edges, we have that $f(e) = 0$ or 
$f(e') = 0$.  An A-multidissection has $k$ edges if 
$\sum_{e \in E_A} f(e) = k$.
Figure 2.1 shows an A-multidissection of $\mathbb{P}_9$ with six edges.
For $k$ fixed, the set of A-multidissections
of $\mathbb{P}_n$ with $k$ edges carries an action of rotation.  

The simplicial complex $\Delta^A_n$ of A-multidissections of 
$\mathbb{P}_n$ is closely related to the cluster complex 
$\Delta(\Phi_{A_{n-3}})$
of type 
A$_{n-3}$.  In particular, if $\Delta$ is any simplicial complex
on the ground set $V$, let $M(\Delta)$ be the associated multicomplex
whose faces are multisets of the form $\{ v_1^{a_1} , \dots, v_m^{a_m} \}$
where $a_1, \dots, a_m \geq 0$ and $\{ v_1 , \dots, v_m \} \subseteq V$ is 
a face of $\Delta$.
Using slightly nonstandard notation, 
denote by $2^{[n]}$ the $(n-1)$-dimensional simplex of all 
subsets of $[n]$ (this is \emph{not} the vertex set of the $n$-dimensional
hypercube).
Recall that the \emph{join} 
$\Delta \star \Delta'$
of two simplicial complexes
$\Delta$ and $\Delta'$ on the ground sets $V$ and $V'$ is the 
simplicial complex whose ground set is the disjoint union
$V \uplus V'$ and whose faces are disjoint unions $F \uplus F'$ of 
faces $F \in \Delta$ and $F' \in \Delta'$.  
The fact that the $n$ boundary edges of $\mathbb{P}_n$ never occur in
a crossing implies that
the complex $\Delta^A_n$ decomposes as 
a join $\Delta^A_n \cong M(\Delta(\Phi_{A_{n-3}})) \star M(2^{[n]})$.

For $x = (x_{ij})_{1 \leq i \leq n, 1 \leq j \leq 2}$ be an $n \times 2$-matrix
of variables and let $\mathbb{C}[x]$ be the polynomial ring in these 
variables over $\mathbb{C}$.  Let $\mathcal{A}_n$ be the subalgebra
of $\mathbb{C}[x]$ generated by the $2 \times 2$ minors of the matrix
$x$.
For any edge $e = ( i , j ) \in E_A$ with $i < j$, let $z_e^A := \Delta_{ij} \in
\mathcal{A}_n$ be the associated $2 \times 2$ matrix minor.  Given an
A-multidissection $f: E_A \rightarrow \mathbb{N}$, define
$z_f^A \in \mathcal{A}_n$ by
\begin{equation*}
z_f^A := \prod_{e \in E_A} (z_e^A)^{f(e)}.
\end{equation*}
For example, if $f$ is the A-multidissection of $\mathbb{P}_9$ in
Figure 2.1, then 
\begin{equation*}
z_f^A = \Delta_{15} \Delta_{19}^2 \Delta_{35} \Delta_{56} \Delta_{68}.
\end{equation*} 
The ring $\mathcal{A}_n$ is graded by polynomial degree.  Since the 
generating minors of $\mathcal{A}_n$ all have degree 2, we can write
$\mathcal{A}_n \cong \bigoplus_{k \geq 0} V^A(n,k)$, where $V^A(n,k)$ is the 
subspace of $\mathcal{A}_n$ with homogeneous polynomial degree $2k$.  

\begin{thm}
Let $n \geq 3$ and $k \geq 0$.  The set $\{ z_f^A \}$, where $f$ ranges over all A-multidissections
of $\mathbb{P}_n$ with exactly $k$ edges, is a $\mathbb{C}$-basis for the space
$V^A(n,k)$.
\end{thm}

This result dates back to Kung and Rota \cite{KungRota}.  Later work of Pylyavskyy
\cite[Theorem 24]{PN} 
on noncrossing tableaux
implies this result, as well.  Finally, this result was reproven by Fomin and Zelevinsky
\cite[Example 12.6]{FZClusterII} and vastly generalized in the context of cluster algebras.  
Explicitly, Fomin and Zelevinsky endow the ring $\mathcal{A}_n$
with the structure of a type A$_{n-3}$ cluster algebra such that the set
$\{ z_f^A \}$ of products of minors corresponding to A-multidissections
are the cluster monomials.
The general linear group $GL_n(\mathbb{C})$ acts on the matrix 
$x$ of variables by left multiplication.  This gives 
the graded ring $\mathcal{A}_n$ the structure of a graded
$GL_n(\mathbb{C})$-module.

Theorem 2.1 can be used to determine the isomorphism type of 
the components $V^A(n,k)$ of
the graded module $\mathcal{A}_n$.  We will need the notion due to
Pylyavskyy \cite{PN} of a \emph{seminoncrossing tableau}.  Two intervals
$[a,b]$ and $[c,d]$ in $\mathbb{N}$ 
with $a < b$ and $c < d$
are said to be \emph{noncrossing} if neither $a < c < b < d$ nor 
$c < a < d < b$ hold.  Call a rectangular tableau $T$ with 
exactly two rows 
\emph{seminoncrossing} if it entries increase strictly down columns
and its columns are pairwise noncrossing when viewed as intervals in
$\mathbb{N}$.  We consider two seminoncrossing tableaux to be the 
same if they differ only by a permutation of their columns.

\begin{prop} \cite[Theorem 13]{PN}  For $n \geq 0$ and a 
rectangular
partition
$\lambda$ with 
exactly
two rows, there is a content preserving bijection
between seminoncrossing tableaux of shape $\lambda$ and entries
$\leq n$ and semistandard tableaux of shape $\lambda$ and entries
$\leq n$.
\end{prop}

Pylyavskyy gives a more general definition of a seminoncrossing tableau
which applies to general shapes with more than two rows under
which this result still holds, but we 
will not need this here.  In both this special case and the more
general context, the proof of this result is combinatorial and uses the 
fact that every even length Yamanouchi word
\footnote{A \emph{Yamanouchi word} is a finite sequence
$w_1 \dots w_n$ of positive integers such that for all 
$i > 0$ and all $1 \leq j \leq n$, the number of $i$'s
in the prefix $w_1 \dots w_j$ is greater than or equal to the number
of $(i+1)$'s in the prefix $w_1 \dots w_j$.} 
on the letters $1$ and $2$ 
corresponds to a unique 
standard tableau and a unique seminoncrossing tableau of content
$(1,\dots,1)$.  
Since the condition for a tableau to be semistandard can be phrased 
as a nonnesting condition on its columns, Proposition 2.2 can be viewed as 
an instance of combinatorial `duality' between noncrossing
and nonnesting objects.

\begin{lem}
Let $n \geq 3$ and $k \geq 0$.  The 
graded component
$V^A(n,k)$ is isomorphic as a $GL_n(\mathbb{C})$-module to the 
irreducible polynomial
representation of $GL_n(\mathbb{C})$ of highest weight $(k,k)$.
\end{lem}
\begin{proof}
To prove the claimed module isomorphism, we compute the Weyl character
of $V^A(n,k)$.  Let 
$h = \mathrm{diag}(y_1, \dots, y_n) \in GL_n(\mathbb{C})$ be an element
of the Cartan subgroup of diagonal matrices.  For any
A-multidissection $f$ of $\mathbb{P}_n$, the polynomial $z_f^A$ is an eigenvector for the action of 
$h$ on $V^A(n,k)$ with eigenvalue equal to 
\begin{equation*}
\prod_{ij \in {[n] \choose 2}} (y_i y_j)^{f(ij)}.
\end{equation*}
Summing over A-multidissections $f$
and applying Theorem 2.1, the trace of the action of $h$ on $V^A(n,k)$ is equal to
\begin{equation*}
\sum_f \prod_{ij \in {[n] \choose 2}} (y_i y_j)^{f(ij)}.
\end{equation*}
There is an obvious bijection between the set of A-multidissections $f$ of $\mathbb{P}_n$ 
with $k$ edges 
and 
seminoncrossing tableaux 
of shape $(k,k)$ and entries bounded above by $n$ obtained by 
letting the edge $(i,j)$ with $i < j$ correspond to the length two column containing $i$ above
$j$.  
For example, the A-multidissection in Figure 2.1 is mapped to the seminoncrossing 
tableau:
$\begin{array}{cccccc}
1 & 1 & 1 & 3 & 5 & 6 \\
5 & 9 & 9 & 5 & 6 & 8
\end{array}$.
(When drawing A-multidissections, edges are drawn with the multiplicity
given by the multidissection and dashed boundary edges indicate
boundary edges which are included with multiplicity zero.)

\begin{figure}
\centering
\includegraphics[scale=.6]{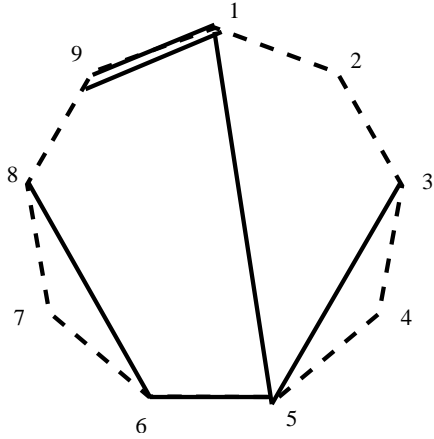}
\caption{An A-multidissection of $\mathbb{P}_9$ with six edges}
\end{figure}

Since this bijection preserves weights,
Proposition 2.2 implies that the above expression is equal to the Schur function
\begin{equation*}
s_{(k,k)}(y_1, \dots, y_n),
\end{equation*}
which proves the desired module isomorphism.  
\end{proof}

Define $g^A$ to be the element of $GL_n(\mathbb{C})$ given by
\begin{equation*}
g^A = \begin{pmatrix}
0 & 0 & 0 & \dots & 0 & -1 \\
1 & 0 & 0 & \dots & 0 & 0 \\
0 & 1 & 0 & \dots & 0 & 0 \\
0 & 0 & 1 & \dots & 0 & 0 \\
   &     &    & \dots &    &     \\
0 & 0 & 0 & \dots & 1 & 0
\end{pmatrix}.
\end{equation*}
Thus, $g^A$ is the permutation matrix for the long cycle in $\mathfrak{S}_n$ with the upper
right $1$ replaced with a $-1$.  The following observation is a simple computation.

\begin{obs}
Let $r$ denote the rotation operator and let $f$ by any A-multidissection of $\mathbb{P}_n$.  Then,
\begin{equation*}
g^A . z_f^A = z_{r.f}^A.
\end{equation*}
\end{obs}

\begin{thm}
Fix $n \geq 3$ and $k \geq 0$ and let $X$ be the set of A-multidissections of $\mathbb{P}_n$ with 
$k$ edges.  Let the cyclic group $C = \mathbb{Z}_n$ of order $n$ act on $X$ by rotation.  The triple
$(X, C, X(q))$ exhibits the cyclic sieving phenomenon, where
\begin{equation*}
X(q) = q^{-k} s_{(k,k)}(1,q,\dots,q^{n-1}).
\end{equation*}
\end{thm}

\begin{proof}
Fix $d \geq 0$ and let $\zeta = e^{\frac{2 \pi i}{n}}$.  Let $r: X \rightarrow X$ be the 
rotation operator.
By Theorem 2.1, we know that the set $\{ z_f^A \,|\, f \in X \}$ forms a basis for
$V^A(n,k)$.  Observation 2.4 implies that the fixed point set cardinality 
$|X^{r^d}|$ is equal to the trace of $(g^A)^d$ on $V^A(n,k)$. 

On the other hand, we have that $(g^A)^d$ is $GL_n(\mathbb{C})$-conjugate to
the diagonal matrix 
$\alpha^{-d} \mathrm{diag}(1, \zeta^d, \dots, \zeta^{(n-1)d})$, where
$\alpha = e^{\frac{\pi i}{n}}$.  So, by Lemma 2.3, the trace of $(g^A)^d$ on
$V^A(n,k)$ is the Schur function specialization
$\alpha^{-2dk} s_{(k,k)}(1,\zeta^d, \dots, \zeta^{(n-1)d}) = X(\zeta^d)$,
as desired. 
\end{proof}

Fixing $n$ and a divisor $d | n$, the set of A-multidissections
of $\mathbb{P}_n$ which are fixed by the $d$-th power of rotation is closely
related to the set of classical dissections of $\mathbb{P}_n$ which
are fixed by the $d$-th power of rotation.  
An enumerative sieve argument implies that the collection of fixed point
set cardinalities given by Part 1 of Theorem 1.1 is obtainable from the
set of fixed point set cardinalities given in Theorem 2.5, and vice versa.

The polynomial $X(q)$ appearing in Theorem 2.5 has a nice product
formula given by Stanley's $q$-hook content formula \cite{StanEC2}.  
$X(q)$ is also the generating function for plane partitions inside 
a $2 \times k \times (n-2)$-box.  The polynomial $X(q)$ appears in
a CSP involving the action of promotion on semistandard tableaux of
shape $(k,k)$ and entries bounded above by $n$ \cite{RCSP}.  
The
$q$-shift $q^k X(q) = [s_{(k,k)}(x_1, \dots, x_n)]_{x_i = q^{i-1}}$ is
called the \emph{principal specialization} of the Schur function
$s_{(k,k)}(x_1, \dots, x_n)$.  Finally, the polynomial $X(q)$ can be 
interpreted 
(up to $q$-shift)
as the $q$-Weyl dimension formula for the irreducible
polynomial representation of $GL_n(\mathbb{C})$ of highest weight
$(k,k)$.

\section{Type B/C}

While the actions of the 
deformed Coxeter element $\tau$ on the cluster complexes
of types B and C are identical, the geometric realizations of the cluster
algebras of types B and C given in \cite{FZClusterII} are quite different.  
In proving our multiplicity counting analog of Part 2 of Theorem 1.1 we
will use the geometric realization of the
type C cluster algebra in \cite{FZClusterII}.

For $n \geq 2$, a \emph{C-edge} in $\mathbb{P}_{2n}$ is either a pair of
centrally symmetric nondiameter edges (which may be on the boundary of
$\mathbb{P}_{2n}$) or a diameter of $\mathbb{P}_{2n}$.  A
\emph{C-multidissection} of $\mathbb{P}_{2n}$
is a function $f: E_C \rightarrow \mathbb{N}$, 
where $E_C$ is the set of C-edges in $\mathbb{P}_{2n}$, such that for any pair
$e, e'$ of crossing C-edges we have $f(e) = 0$ or $f(e') = 0$.  As in the
type A case, 
the fact that the $n$ pairs of centrally symmetric boundary edges of 
$\mathbb{P}_{2n}$ never occur in a crossing implies that
the simplicial complex $\Delta^C_n$ formed by the 
C-multidissections of $\mathbb{P}_{2n}$ is related to the type C$_{n-1}$ cluster
complex $\Delta(\Phi_{C_{n-1}})$ via
$\Delta^C_n \cong M(\Delta(\Phi_{C_{n-1}})) \star M(2^{[n]})$.  A 
C-multidissection $f$ has $k$ edges if $\sum_{e \in E_C} f(e) = k$.
Figure 3.1 shows a C-multidissection of $\mathbb{P}_8$ with five edges.
Rotation acts with order $n$ on the set of C-multidissections of 
$\mathbb{P}_{2n}$ with $k$ edges.

\begin{figure}
\centering
\includegraphics[scale=.6]{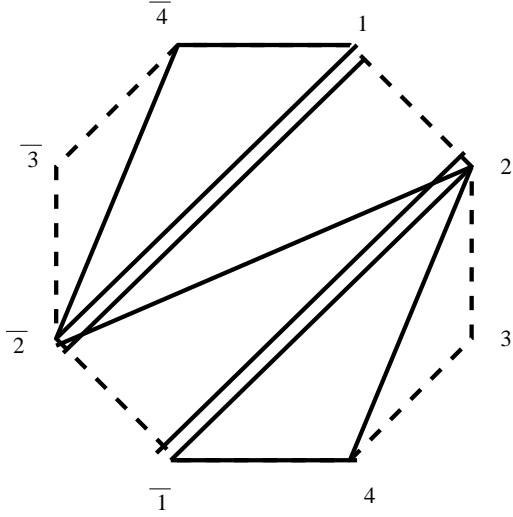}
\caption{A C-multidissection of $\mathbb{P}_8$ with five edges}
\end{figure}

We recall the geometric realization of the type C$_{n-1}$ cluster algebra
given in \cite{FZClusterII}.  Fix $n \geq 2$, let 
$x = (x_{ij})_{1 \leq i \leq n, 1 \leq j \leq 2}$ be an $n \times 2$ matrix
of variables, and let $\mathbb{C}[x]$ be the polynomial ring in these
variables over $\mathbb{C}$.  The multiplicative group $\mathbb{C}^{\times}$
of nonzero complex numbers acts on $\mathbb{C}[x]$ by
$\alpha . x_{i1} = \alpha x_{i1}$ and
$\alpha . x_{i2} = \alpha^{-1} x_{i2}$ for $1 \leq i \leq n$.  Let
$\mathbb{C}[x]^{\mathbb{C}^{\times}}$ be the invariant subalgebra for 
this action.  Since no polynomial in $\mathbb{C}[x]$ containing a nonzero
homogeneous component of odd degree is fixed by $\mathbb{C}^{\times}$, we 
have the grading $\mathbb{C}[x]^{\mathbb{C}^{\times}} \cong \bigoplus_{k \geq 0}
V^C(n,k)$, where $V^C(n,k)$ is the subspace of 
$\mathbb{C}[x]^{\mathbb{C}^{\times}}$ spanned by polynomials of homogeneous
degree $2k$.

We construct a graded action of $GL_n(\mathbb{C})$ on 
$\mathbb{C}[x]^{\mathbb{C}^{\times}}$  as follows.  Let 
$\sigma: GL_n(\mathbb{C}) \rightarrow GL_n(\mathbb{C})$ be the
involution defined by $\sigma(A) = (A^{-1})^T$.  The group 
$GL_n(\mathbb{C})$ acts on the polynomial algebra
$\mathbb{C}[x_{11}, \dots, x_{n1}]$ by considering this algebra as 
the symmetric algebra over the defining representation of 
$GL_n(\mathbb{C})$.  In addition, the group $GL_n(\mathbb{C})$ acts 
on the algebra $\mathbb{C}[x_{12}, \dots, x_{n2}]$ via the above 
action precomposed with $\sigma$.  These actions tensor together to 
give an action of $GL_n(\mathbb{C})$ on 
$\mathbb{C}[x] \cong \mathbb{C}[x_{11}, \dots, x_{n1}] \otimes_{\mathbb{C}}
\mathbb{C}[x_{12}, \dots, x_{n2}]$.  Restriction of this action to the
center $Z(GL_n(\mathbb{C})) \cong \mathbb{C}^{\times}$ of nonzero
multiples of the identity matrix yields the action of $\mathbb{C}^{\times}$
in the above paragraph.  Therefore, the invariant space
$\mathbb{C}[x]^{\mathbb{C}^{\times}}$ is a graded $GL_n(\mathbb{C})$-module.
As a $GL_n(\mathbb{C})$-module, we
have that 
$\mathbb{C}[x] \cong Sym(\mathbb{C}^n) \otimes_{\mathbb{C}} 
Sym((\mathbb{C}^n)^*)$, where $\mathbb{C}^n$ carries the defining 
representation of $GL_n(\mathbb{C})$ and $(\mathbb{C}^n)^*$ carries
the dual of the defining representation of $GL_n(\mathbb{C})$.  The 
invariant subalgebra $\mathbb{C}[x]^{\mathbb{C}^{\times}}$ 
is the Serge subalgebra
of $\mathbb{C}[x]$ consisting of polynomials which are bihomogeneous
with degrees of the form $(k,k)$.  Therefore, we have the  
decomposition of $GL_n(\mathbb{C})$-modules:
\begin{equation*}
\mathbb{C}[x]^{\mathbb{C}^{\times}} = \bigoplus_{k \geq 0} 
Sym^k(\mathbb{C}^n) \otimes_{\mathbb{C}} Sym^k((\mathbb{C}^n)^*).
\end{equation*}
The next result follows from looking at the $k$-th graded
piece of the above direct sum.

\begin{lem}
Let $n \geq 2$ and $k \geq 0$.  
Let $\mathbb{C}^n$ carry the defining representation of $GL_n(\mathbb{C})$ and let 
$(\mathbb{C}^n)^{*}$ be the dual of this representation.
The $GL_n(\mathbb{C})$-module $V^C(n,k)$ is isomorphic to the rational
$GL_n(\mathbb{C})$-module 
$Sym^k(\mathbb{C}^n) \otimes_{\mathbb{C}} Sym^k((\mathbb{C}^n)^{*})$.
\end{lem}

As in \cite[Example 12.12]{FZClusterII}, C-multidissections can be 
used to build
$\mathbb{C}$-bases for the spaces $V^C(n,k)$.  Label the vertices 
of $\mathbb{P}_{2n}$ clockwise with $1, 2, \dots, n, \bar{1}, \bar{2}, \dots,
\bar{n}$.  For any C-edge $e \in E_C$ we associate a polynomial
$z_e^C \in V^C(n,1)$ as follows.  If $e$ is a diameter of the form
$a \bar{a}$ for $1 \leq a \leq n$, define $z_e^C = x_{a1} x_{a2}$.  If $e$
is a pair of `integrated' centrally symmetric nondiameter edges of the form
$a \bar{b}, \bar{a} b$ for $1 \leq a < b \leq n$, define
$z_e^C = \frac{x_{a1} x_{b2} + x_{a2} x_{b1}}{2}$.  Finally, if 
$e$ is a pair of `segregated' centrally symmetric nondiameter edges 
of the form $ab, \bar{a}\bar{b}$ for $1 \leq a < b \leq n$, define
$z_e^C = \frac{x_{a1}x_{b2} - x_{a2} x_{b1}}{2i}$.  If 
$f: E_C \rightarrow \mathbb{N}$ is any C-multidissection of
$\mathbb{P}_{2n}$, define $z_f^C$ by
\begin{equation*}
z_f^C := \prod_{e \in E_C} (z_e^C)^{f(e)}.
\end{equation*}
For example, if $f$ is the C-multidissection of $\mathbb{P}_8$ shown in
Figure 3.1, then
\begin{equation*}
z_f^C = (\frac{x_{11}x_{22} + x_{12}x_{21}}{2})^2(\frac{x_{11}x_{24}+x_{14}
x_{21}}{2})(\frac{x_{12}x_{24}-x_{14}x_{22}}{2i})(x_{21}x_{22}) \in
V^C(4,5).
\end{equation*}
The type C analog of Theorem 2.1 is as follows.

\begin{thm}\cite[Proposition 12.13]{FZClusterII}
Let $n \geq 2$ and $k \geq 0$.  Then, the set $\{ z^C_f \}$, where $f$ ranges over all 
C-multidissections of $\mathbb{P}_{2n}$ with exactly $k$ edges, forms a 
$\mathbb{C}$-basis for $V^C(n,k)$.
\end{thm}

Let $g^C$ be the element of $GL_n(\mathbb{C})$ given by
\begin{equation*}
g^C = \begin{pmatrix}
0 & 0 & 0 & \dots & 0 & i \\
1 & 0 & 0 & \dots & 0 & 0 \\
0 & 1 & 0 & \dots & 0 & 0 \\
0 & 0 & 1 & \dots & 0 & 0 \\
   &     &    & \dots &     &   \\
0 & 0 & 0 & \dots & 1 & 0 
\end{pmatrix}.
\end{equation*}
Thus, $g^C$ is the permutation matrix for the long cycle in $\mathfrak{S}_n$, except that 
the upper right hand entry is $i$ instead of $1$.  The following observation
is a direct computation involving the action of $g^C$ on $z_e^C$ in the 
cases where $e$ is a diameter or a pair of centrally symmetric nondiameter
edges.  This latter case breaks up into two subcases depending on whether
the vertex $n$ is involved in the edges in the C-edge $e$.

\begin{obs}
Let $r$ denote the rotation operator and let $f$ be any C-multidissection of $\mathbb{P}_{2n}$.  
Then,
\begin{equation*}
g^C . z_f^C = z_{r.f}^C.
\end{equation*}
\end{obs}

\begin{thm}
Let $n \geq 2$ and $k \geq 0$.  Let $X$ be the set of C-multidissections of 
$\mathbb{P}_{2n}$ with $k$ edges and let the cyclic group $C = \mathbb{Z}_{n}$ act on 
$X$ by rotation.  The triple $(X, C, X(q))$ exhibits the cyclic sieving phenomenon, where
\begin{equation*}
X(q) = h_{(k,k)}(1,q,\dots,q^{n-1}).
\end{equation*}
\end{thm}

\begin{proof}
Let $r: X \rightarrow X$ be the rotation operator and fix $d \geq 0$.  By Theorem 3.2 and 
Observation 3.3, the cardinality of the fixed point set $X^{r^d}$ is equal to the trace of the 
linear operator $(g^C)^d$ on the $GL_n(\mathbb{C})$-module $V^C(n,k)$.
We have that $(g^C)^d$ is $GL_n(\mathbb{C})$-conjugate to the diagonal matrix
$\alpha^d \mathrm{diag}(1,\zeta^d,\dots,\zeta^{(n-1)d}$), where $\zeta = e^{\frac{2 \pi i}{n}}$ and
$\alpha = e^{\frac{\pi i}{2n}}$.  Lemma 3.1 implies that the trace in question is equal to
\begin{equation*}
\alpha^d h_k(1,\zeta^d,\dots,\zeta^{(n-1)d}) \alpha^{-d} 
h_k(1,\zeta^{-d}, \dots, \zeta^{(1-n)d}) = h_{(k,k)}(1,\zeta^d, \dots, \zeta^{(n-1)d}) = X(\zeta^d),
\end{equation*}
as desired.
\end{proof}

As with Theorem 2.5, an enumerative sieve can be
used to relate the fixed point set sizes 
predicted by Theorem 3.4 to those predicted by Part 2 of Theorem 1.1.

The polynomial $h_k(1,q,\dots,q^{n-1})$ is the $q$-analog of a
multinomial coefficient, and therefore the polynomial $X(q)$ in
Theorem 3.4 has a nice product formula.  To our knowledge this is the 
first occurrence of the polynomial
$h_{(k,k)}(1,q,\dots,q^{n-1})$ in a CSP.

\section{Type D}

For $n \geq 2$, a \emph{D-edge} of $\mathbb{P}_{2n}$ is either a pair of 
centrally symmetric nondiameter edges of $\mathbb{P}_{2n}$ (which 
may be boundary edges) or a diameter of $\mathbb{P}_{2n}$ colored one
of two colors, solid/blue or dotted/red.  As in Section 1, two D-edges are said to 
\emph{cross} if they cross in the classical sense, except that
distinct diameters of the same color and identical diameters of different
colors do not cross.  A \emph{D-multidissection} of $\mathbb{P}_{2n}$ is a
function $f : E_D \rightarrow \mathbb{N}$, where $E_D$ is the 
set of D-edges in $\mathbb{P}_{2n}$ and whenever $e$ and $e'$ are crossing
D-edges, we have $f(e) = 0$ or $f(e') = 0$.  Also as in Section 1,
\emph{D-rotation} acts on the set of D-multidissections by standard rotation,
except that D-rotation switches the colors of the diameters.  

When counting
the edges in a D-multidissection, we count diameters as one edge and pairs
of centrally symmetric nondiameter edges as two edges.  More formally,
if $E_D(CS)$ denotes the set of centrally symmetric pairs of nondiameters
in $\mathbb{P}_{2n}$ and if $E_D(D)$ denotes the set of 
colored
diameters in
$\mathbb{P}_{2n}$, we say that a D-multidissection $f$ 
of $\mathbb{P}_{2n}$
has $k$ edges if
$\sum_{e \in E_D(CS)} 2f(e) + \sum_{e \in E_D(D)} f(e) = k$.
For example, the D-multidissection of $\mathbb{P}_8$ on the right in
Figure 4.1 has nine edges.  
This counting convention \emph{differs} from that of Part 3 of Theorem
1.1, where a pair of centrally symmetric nondiameters counted as a single 
edge.
Although this
counting convention is less natural from the point of view of cluster
complexes in which a pair of centrally symmetric nondiameters corresponds
to a single almost positive root, it is more natural from the
standpoint of edge enumeration in polygon dissections.

The proof of our CSP involving D-multidissections will use both
algebraic and combinatorial methods.
Recall the definition of the algebra $\mathcal{A}_n$ from Section
2.  We define the algebra $V^D(n)$ to be the quotient of $\mathcal{A}_{n+2}$
by the 
ideal generated by the
minor $\Delta_{n+1,n+2}$.  For $1 \leq i \leq n+2$ and 
$1 \leq j \leq 2$, the \emph{D-degree} of the variable 
$x_{ij} \in \mathcal{A}_{n+2}$ is defined to be 1 if 
$1 \leq i \leq n$ and 0 otherwise.  
The D-degree of a $2 \times 2$-minor $\Delta_{ij}$
remains well-defined in the
quotient ring $V^D(n)$ and induces a grading 
$V^D(n) \cong \bigoplus_{k \geq 0} V^D(n,k)$, where $V^D(n,k)$ is the 
subspace of $V^D(n)$ with homogeneous D-degree $k$.  We can use 
A-multidissections to write down a basis for $V^D(n,k)$.
Viewing A-multidissections as multidissections of edges, any
A-multidissection of $\mathbb{P}_{n+2}$ gives rise to a multiset 
of endvertices. 
If we label the vertices of $\mathbb{P}_{n+2}$ clockwise with 
$1, 2, \dots, n+2$,
it makes sense to count how many of these
vertices lie in $[n]$ with multiplicity.  For example, if $n = 7$,
we have that $9$ vertices of the A-multidissection of
$\mathbb{P}_9$ shown in Figure 2.1 lie in $[7]$ counting multiplicity.

\begin{lem}
Let $n \geq 2$ and $k \geq 0$.  Abusing notation, identify polynomials in
$\mathcal{A}_{n+2}$ with their images in $V^D(n)$.  We have that the set
$\{z^A_f \}$, where $f$ ranges over all A-multidissections of 
$\mathbb{P}_{n+2}$ with $f(n+1,n+2) = 0$ and $k$ vertices of $f$ 
counting multiplicity are in $[n]$, is a $\mathbb{C}$-basis
for $V^D(n,k)$.
\end{lem}
\begin{proof}
Let $J \subset \mathcal{A}_{n+2}$ be the ideal generated by the minor
$\Delta_{n+1,n+2}$.  By Theorem 2.1 it is enough to show that $J$ is
spanned over $\mathbb{C}$ by all polynomials of the form 
$z^A_f$, where $f$ is an A-multidissection of $\mathbb{P}_{n+2}$
satisfying $f(n+1,n+2) > 0$.  Clearly this span is contained in $J$.  
The reverse containment follows from the fact that the boundary 
edge $(n+1,n+2)$ crosses none of the edges in $\mathbb{P}_{n+2}$, and
can therefore be added to any A-multidissection without creating any 
crossings.
\end{proof}

As in Section 2, the space $\mathcal{A}_{n+2}$ is a 
$GL_{n+2}(\mathbb{C})$-module.  Considering the inclusion
$GL_n(\mathbb{C}) \times GL_2(\mathbb{C}) \subset GL_{n+2}(\mathbb{C})$ via
$(A,B) \mapsto \begin{pmatrix} A & 0 \\ 0 & B \end{pmatrix}$, 
the space $\mathcal{A}_{n+2}$ is also a $GL_n(\mathbb{C}) \times 
GL_2(\mathbb{C})$-module by restriction.
This latter action descends to the quotient $V^D(n)$ and respects
D-degree, giving $V^D(n) \cong \bigoplus_{k \geq 0} V^D(n,k)$ the structure
of a graded $GL_n(\mathbb{C}) \times GL_2(\mathbb{C})$-module.  The
isomorphism type of $V^D(n,k)$ can be determined using Lemma 4.1.

\begin{lem}
Let $n \geq 2$ and $k \geq 0$.  
For any partition $\lambda$, let $V_{\lambda}$ be the irreducible polynomial representation of
$GL_n(\mathbb{C})$ of highest weight $\lambda$.
The $GL_n(\mathbb{C}) \times GL_2(\mathbb{C})$-module
$V^D(n,k)$ is isomorphic to the representation of $GL_n(\mathbb{C}) \times GL_2(\mathbb{C})$ given by 
$\bigoplus_{\ell=0}^{\lfloor \frac{k}{2} \rfloor} 
V_{(k-\ell,\ell)} \otimes_{\mathbb{C}} Sym^{k-2\ell}(\mathbb{C}^2)$,
where $\mathbb{C}^2$ carries the defining representation of $GL_2(\mathbb{C})$.
\end{lem} 

\begin{proof}
Let $\mathcal{A}_{n+2}(-z_1 z_2)$ denote
the $GL_n(\mathbb{C}) \times GL_2(\mathbb{C})$-module whose underlying 
vector space is the same as
$\mathcal{A}_{n+2}$ and whose module structure is obtained
by $(g,h) \cdot v := \mathrm{det}(h) (g,h).v$, where the raised dot
denotes action on $\mathcal{A}_{n+2}(-z_1 z_2)$ and the lowered dot 
denotes action on $\mathcal{A}_{n+2}$.  Lemma 4.1 implies that we 
have the following short exact sequence of $GL_n(\mathbb{C}) \times
GL_2(\mathbb{C})$-modules.  
\begin{equation*}
0 \rightarrow \mathcal{A}_{n+2}(-z_1 z_2) 
\xrightarrow{ \cdot \Delta_{n+1,n+2}} \mathcal{A}_{n+2} \longrightarrow
V^D(n) \rightarrow 0,
\end{equation*}
where the left hand map is multiplication by $\Delta_{n+1,n+2}$ and
the right hand map is the canonical projection. This short exact sequence
implies that the trace
of the action of $(\mathrm{diag}(y_1, \dots y_n), \mathrm{diag}(z_1,z_2)) 
\in GL_n(\mathbb{C}) \times GL_2(\mathbb{C})$ on $V^D(n)$ is equal to
\begin{equation*}
(1 - z_1 z_2) \sum_{m \geq 0} s_{(m,m)}(y_1, \dots, y_n, z_1, z_2) = 
\sum_T (yz)^T,
\end{equation*}
where the latter sum ranges over all semistandard tableaux $T$ of shape
$(m,m)$ and entries $y_1 < \dots < y_n < z_1 < z_2$ having no
$z$'s in the last column.  Restricting such a tableau $T$ to its $y$
and $z$ entries gives a semistandard tableau $T_1$ of some shape
$(m,\ell)$ in the $y$ entries and another semistandard tableau $T_2$ of
shape $(m - \ell)$ in the $z$ entries.  Such a tableau $T$ therefore
contributes a term to the Weyl character of $V^D(n,k)$, where $k = m+\ell$.
We have that $m - \ell = k - 2 \ell$, so that $T$ contributes a typical
term of $s_{(k - \ell, \ell)}(y_1, \dots, y_n) h_{k-2\ell}(z_1,z_2)$.
It follows that the Weyl character of $V^D(n,k)$ agrees with the Weyl 
character of the module in the statement of the lemma.
\end{proof} 
We remark that a combinatorial proof of Lemma 4.2 can be obtained
using Lemma 4.1 and an explicit weight-preserving bijection between
A-multidissections and seminoncrossing tableaux which computes the
Weyl character of $V^D(n,k)$.  A more general notion (due to
Pylyavskyy \cite{PN}) of seminoncrossing
tableaux for two-row shapes which are not rectangular is needed in this context. 

The module isomorphism in Lemma 4.2 can be combined with a geometric
realization of the type D cluster algebras in \cite{FZClusterII} to give
an enumeration of the $k$-edge D-multidissections of 
$\mathbb{P}_{2n}$ in terms of symmetric function evaluations.
In what follows, if $f(x_1, \dots, x_n)$ is a symmetric function
in $n$ variables, abbreviate
by $f(1^n)$ the evaluation $f(1, \dots, 1)$.  Also, adopt
the convention that the digon
$\mathbb{P}_2$ has two types of D-edges:  a `diameter' which 
can be colored solid/blue or dotted/red.  In accordance with our 
earlier
conventions regarding
crossings in D-multidissections, no multiset of D-edges in 
$\mathbb{P}_2$ has a crossing.  We count each diameter in 
$\mathbb{P}_2$ as a \emph{single} edge, so that the number of 
$k$-edge D-multidissections of $\mathbb{P}_2$ is equal to $k+1$ for
all $k$.

\begin{cor}
Let $n \geq 1$ and $k \geq 0$.  The number of $k$-edge D-multidissections
of $\mathbb{P}_{2n}$ is equal to
$2 h_{(k)}(1^n) + 2 h_{(k-1,1)}(1^n) + \cdots + 2 h_{(\frac{k+1}{2}, \frac{k-1}{2})}(1^n)$
if $k$ is odd, and
$2 h_{(k)}(1^n) + 2 h_{(k-1,1)}(1^n) + \cdots + 2 h_{(\frac{k}{2}+1, \frac{k}{2}-1)}(1^n) + h_{(\frac{k}{2},\frac{k}{2})}(1^n)$ if $k$ is even.
\end{cor}
\begin{proof}
If $n = 1$ this enumeration of D-multidissections of the digon can be
checked directly.  For $n > 1$, the geometric realization of the type
D cluster algebra given in \cite[Section 12.4]{FZClusterII} gives a 
cluster monomial
basis for
$V^D(n)$ where cluster variables are indexed by D-edges, clusters monomials
are indexed by D-multidissections, the cluster variables corresponding
to colored diameters have D-degree 1, and the cluster variables corresponding
to pairs of centrally symmetric nondiameters have D-degree 2.  
Therefore, the number of $k$-edge D-multidissections of $\mathbb{P}_{2n}$ is
equal to the dimension of the space $V^D(n,k)$.  This dimension is 
equal to the Weyl character in
the proof of Lemma 4.2 evaluated at the identity matrix
$y_1 = \dots = y_n = z_1 = z_2 = 1$.  For all $k$ and $\ell$ we have that
$h_{k - 2\ell}(1^2) = k - 2 \ell + 1$.  So, the number of $k$-edge
D-multidissections of $\mathbb{P}_{2n}$ is equal to  
\begin{equation*}
\sum_{\ell = 0}^{\lfloor \frac{k}{2} \rfloor}
(k - 2\ell + 1) s_{(k-\ell,\ell)}(1^n).
\end{equation*}
The equality of this expression and the expression in the statement of
the corollary is a consequence of Pieri's Rule.
\end{proof}

The following lemma will be useful in the proof of our type D CSP in
Theorem 4.6.  Roughly, it states that D-multidissections of $\mathbb{P}_{2n}$
which are invariant under a fixed even power of D-rotation are in
bijection 
with D-multidissections of a smaller polygon.  

\begin{lem} (Folding Lemma)
Let $k \geq 0$, $n \geq 2$, and let $d$ be an even divisor of $2n$.  Let $r$
be the D-rotation operator on D-multidissections of $\mathbb{P}_{2n}$. \\
1.  If $d$ divides $n$, the set of $k$-edge $r^d$-invariant D-multidissections of 
$\mathbb{P}_{2n}$ is in bijective correspondence with the set of
$\frac{kd}{n}$-edge D-multidissections of $\mathbb{P}_{2d}$.  \\
2.  If $d$ does not divide $n$, the set of $k$-edge $r^d$-invariant
D-multidissections of $\mathbb{P}_{2n}$ is in bijective correspondence with
the set of $\frac{kd}{2n}$-edge D-multidissections of $\mathbb{P}_d$. \\
We interpret the set of $m$-edge D-multidissections of a polygon to be
empty if $m$ is not an integer.
\end{lem}

Lemma 4.4 does not hold if $d$ is odd.  In this case, the operator
$r^d$ swaps diameter colors and therefore no $r^d$-invariant D-multidissection
of $\mathbb{P}_{2n}$ contains a diameter.  When $d$ is odd, there is
a natural bijection between $k$-edge $r^d$-invariant D-multidissections
of $\mathbb{P}_{2n}$ and $\frac{k}{2}$-edge C-multidissections of 
$\mathbb{P}_{2n}$ which are invariant under the $d$-th power of rotation.

\begin{proof}
Let $C = \langle r^d \rangle$ be the cyclic group generated by $r^d$.  
Since $d$ is even, observe that $r^d$ preserves diameter colors and 
is therefore
the $d$-th power of the classical rotation operator.
The 
idea behind the bijective correspondences in Parts 1 and 2 is to map each
$C$-orbit 
$\mathcal{O}$ of mutually noncrossing
 D-edges in $\mathbb{P}_{2n}$ to a one- or two-element set 
$\psi(\mathcal{O})$ of 
noncrossing D-edges in $\mathbb{P}_{2d}$ (in Part 1) or 
$\mathbb{P}_{d}$ (in Part 2).  
Roughly speaking, $\psi(\mathcal{O})$ will be a pair of diameters of 
different colors when $\mathcal{O}$ is an inscribed polygon and
$\psi(\mathcal{O})$ will be a singleton consisting of a D-edge
obtained by `folding' $\mathbb{P}_{2n}$ otherwise.
This assignment $\mathcal{O} \mapsto \psi(\mathcal{O})$ will induce the desired
bijection $\Psi$ of D-multidissections.

For Part 1, label the vertices of $\mathbb{P}_{2n}$ clockwise with
the ordered pairs 
\begin{equation*}
(1,1), (2,1), \dots, \\ (d,1), (\bar{1},1), (\bar{2},1),
\dots, (\bar{d},1), \dots, (1, \frac{n}{d}), \dots,  (d, \frac{n}{d}), 
(\bar{1}, \frac{n}{d}) ,\dots, (\bar{d}, \frac{n}{d}).  
\end{equation*}
There are four 
types of $C$-orbits of 
mutually noncrossing
$D$-edges in $\mathbb{P}_{2n}$: \\
(1) monochromatic 
sets of diameters 
\begin{equation*}
\{ ((i,j), (\bar{i},j+ \lfloor \frac{n}{2d} \rfloor)) \,|\, 1 \leq j \leq \frac{n}{d} \} 
\end{equation*}
of $\mathbb{P}_{2n}$,
where
the second indices of ordered pairs are interpreted modulo
$\frac{n}{d}$ and $1 \leq i \leq d$, \\
(2) sets of `segregated' nondiameter edges of the form 
\begin{equation*}
\{ 
((i,k), (j,k)), ((\bar{i},k),(\bar{j},k)) \,|\, 1 \leq k \leq \frac{n}{d} \}, 
\end{equation*}
where $1 \leq i < j \leq d$, \\
(3) sets of `integrated' nondiameter edges of the form 
\begin{equation*}
\{ 
((j,k),(\bar{i},k)),((\bar{j},k),(i,k+1) \,|\, 1 \leq k \leq \frac{n}{d} \},
\end{equation*}
where the second indices of the ordered pairs are interpreted modulo 
$\frac{n}{d}$ and $1 \leq i < j \leq d$, and \\
(4) sets of `integrated' nondiameter edges of the form
\begin{equation*}
\{ ((i,k),(\bar{i},k)), ((\bar{i},k),(i,k+1)) \,|\, 1 \leq k \leq 
\frac{n}{d} \},
\end{equation*}
 where the second indices of the ordered pairs are 
interpreted modulo $\frac{n}{d}$ and $1 \leq i \leq d$.
Observe that these edges form an inscribed $\frac{2n}{d}$-gon in
$\mathbb{P}_{2n}$.\\
To every C-orbit $\mathcal{O}$ of mutually noncrossing D-edges in 
$\mathbb{P}_{2n}$ we associate a one- or two-element set $\psi(\mathcal{O})$
of mutually noncrossing D-edges in $\mathbb{P}_{2d}$ as follows. 
Label the vertices of $\mathbb{P}_{2d}$ clockwise with 
$1, 2, \dots, d, \bar{1}, \bar{2}, \dots, \bar{d}$.\\
(1) If $\mathcal{O}$ is a monochromatic set of diameters 
\begin{equation*}
\{ ((i,j), (\bar{i},j+ \lfloor \frac{n}{2d} \rfloor)) \,|\, 1 \leq j \leq \frac{n}{d} \},
\end{equation*}
let $\psi(\mathcal{O})$ be the singleton consisting of the diameter $i \bar{i}$ in 
$\mathbb{P}_{2d}$ which has the same color as the diameters in $\mathcal{O}$. \\
(2) If $\mathcal{O}$ is a set of `segregated' nondiameter edges of the form
\begin{equation*}
\{ 
((i,k), (j,k)), ((\bar{i},k),(\bar{j},k)) \,|\, 1 \leq k \leq \frac{n}{d} \},
\end{equation*}
let $\psi(\mathcal{O})$ be the singleton consisting of the D-edge in 
$\mathbb{P}_{2d}$
which is the pair $ij, \bar{i}\bar{j}$ of centrally symmetric nondiameters. 
\\
(3) If $\mathcal{O}$ is a set of `integrated' nondiameter edges of the form
\begin{equation*}
\{ 
((j,k),(\bar{i},k)),((\bar{j},k),(i,k+1) \,|\, 1 \leq k \leq \frac{n}{d} \},
\end{equation*}
let $\psi(\mathcal{O})$ be the singleton consisting of the D-edge in
$\mathbb{P}_{2d}$ which is the pair $i\bar{j}, j \bar{i}$ of centrally
symmetric nondiameters. \\
(4) If $\mathcal{O}$ is an inscribed polygon of the form
\begin{equation*}
\{ ((i,k),(\bar{i},k)), ((\bar{i},k),(i,k+1)) \,|\, 1 \leq k \leq 
\frac{n}{d} \},
\end{equation*} 
let $\psi(\mathcal{O})$ be the two element set of 
D-edges in $\mathbb{P}_{2d}$ consisting of a solid/blue and dotted/red copy of the 
diameter $i \bar{i}$. \\ 

Given any $C$-orbit $\mathcal{O}$ of mutually noncrossing D-edges, let
$\chi_{\mathcal{O}}$ be the D-multidissection of $\mathbb{P}_{2n}$ 
corresponding to $\mathcal{O}$ and let $\chi_{\psi(\mathcal{O})}$ be the
D-multidissection of $\mathbb{P}_{2d}$ corresponding to 
$\psi(\mathcal{O})$.
Any $C$-invariant D-multidissection $f$ of $\mathbb{P}_{2n}$ can be
written uniquely as a sum 
$f = \sum_{\mathcal{O}} c_{\mathcal{O}} \chi_{\mathcal{O}}$, where
the sum is over all $C$-orbits $\mathcal{O}$ of mutually noncrossing
D-edges and the $c_{\mathcal{O}}$ are nonnegative integers.
It is easy to verify that if the D-edges in two $C$-orbits 
$\mathcal{O}$ and $\mathcal{O}'$ are mutually nonintersecting, then 
the D-edges in the sets $\psi(\mathcal{O})$ and $\psi(\mathcal{O}')$ are 
mutually nonintersecting, as well. 
Thus, the map $\Psi(f) := \sum_{\mathcal{O}} c_{\mathcal{O}} 
\chi_{\psi(\mathcal{O})}$ from the D-edges in $\mathbb{P}_{2d}$ to 
$\mathbb{N}$ is a D-multidissection of $\mathbb{P}_{2d}$.   

One checks 
that $\Psi$ maps $k$-edge D-multidissections to
$\frac{kd}{n}$-edge D-multidissections and that $\Psi$ gives the 
bijective correspondence in Part 1.  
Given a D-multidissection $g$ of $\mathbb{P}_{2d}$, to construct 
$\Psi^{-1}(g)$ one first checks if $g$ contains diameters of both 
colors.  If not, for every D-edge $e$ occurring in $g$, one includes
the edges in the unique $C$-orbit $\mathcal{O}$ with 
$\psi(\mathcal{O}) = \{ e \}$ in $\Psi^{-1}(g)$ with the appropriate 
multiplicity.  If $g$ contains diameters of both colors, one includes the 
the appropriate inscribed $\frac{2d}{n}$-gon in
$\Psi^{-1}(g)$ with multiplicity equal to the minimum of the multiplicities
of the solid/blue and dotted/red diameters in $g$ and then includes the $C$-orbits
corresponding to the remaining D-edges in $g$ in $\Psi^{-1}(g)$
with the appropriate multiplicities.  

An example of the map $\Psi$
for $n = 8$ and $d = 4$ is shown in Figure 4.1.  Repeated edges
correspond to edges counted with multiplicity and dashed edges are 
boundary sides which are counted with multiplicity zero in the 
D-multidissection.  Observe that the inscribed $4$-gon 
in $\mathbb{P}_{16}$
maps to a pair
of solid/blue and dotted/red diameters in $\mathbb{P}_8$.  
The $C$-orbit of
solid/blue diameters in $\mathbb{P}_{16}$ maps to an additional
solid/blue diameter
in $\mathbb{P}_8$.

\begin{figure}
\centering
\includegraphics[scale=.5]{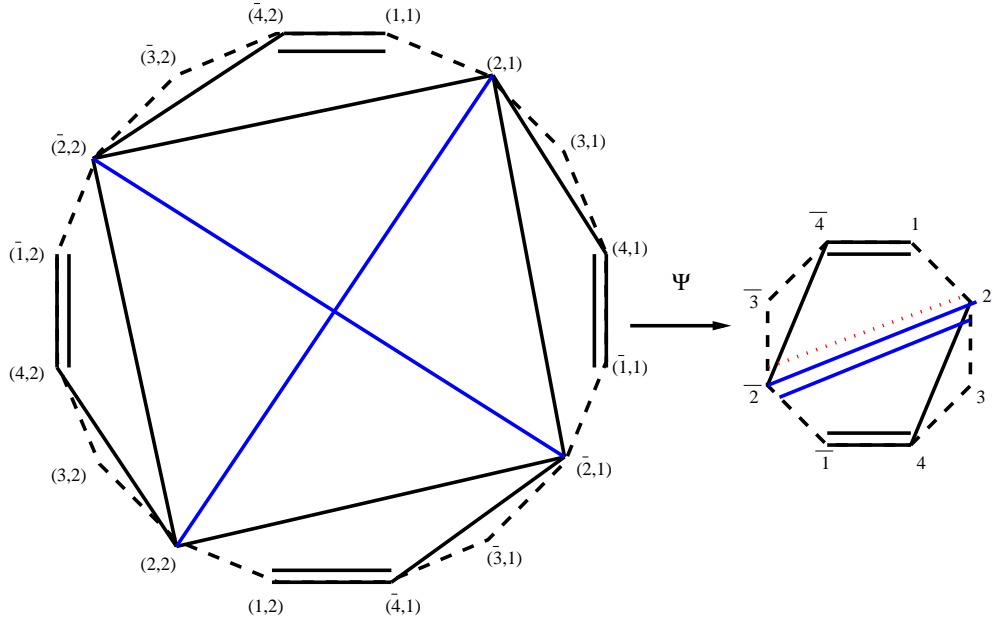}
\caption{The action of the map $\Psi$ on a D-multidissection with
$n = 8$ and $d = 4$}
\end{figure}

The proof of Part 2 mimics the proof of Part 1 and is left to the
reader.  The reason why the polygon $\mathbb{P}_d$ appears in
Part 2 instead of $\mathbb{P}_{2d}$ is that the antipodal image of 
a nondiameter edge $(i,j)$ in $\mathbb{P}_{2n}$ is not contained in
the orbit of $(i,j)$ under $d$-fold rotation if $d$ does not divide $n$.
In the case $d = 2$ one must apply our conventions regarding 
D-multidissections of the digon $\mathbb{P}_2$.    
\end{proof}

We will also need a result on the specialization of homogeneous
symmetric functions at roots of unity which is implicit in a CSP
of Reiner, Stanton, and White.

\begin{lem} \cite[Theorem 1.1, Part a]{RSWCSP}
Let $\zeta$ be a root of unity of order $d$.  If $d | n$ we have the polynomial
evaluation
\begin{equation*}
h_k(1,\zeta, \zeta^2, \dots, \zeta^{n-1}) = \begin{cases}
h_{\frac{k}{d}}(1^{\frac{n}{d}}), & \text{if $d | k$} \\
0 , & \text{otherwise.}
\end{cases}
\end{equation*}
\end{lem}

Our CSP for D-multidissections is as follows.

\begin{thm}
Let $n \geq 2$ and $k \geq 0$.  Let $X$ be the set of D-multidissections of $\mathbb{P}_{2n}$ with 
$k$ edges and let the cyclic group $C = \mathbb{Z}_{2n}$ act on $X$ by D-rotation.  The triple
$(X, C, X(q))$ exhibits the cyclic sieving phenomenon, where
\begin{equation*}
X(q) = 
\sum_{\ell = 0}^{\lfloor \frac{k}{2} \rfloor} s_{(k-\ell,\ell)}(1,q^2,\dots,q^{2(n-1)}) h_{k - 2 \ell}(1,q^n).
\end{equation*}
\end{thm}  
\begin{proof}
Let $r$ be the D-rotation operator and let $\zeta$ be a root of unity of order $2n$.  We equate the 
fixed point sets corresponding to powers 
$r^d$
of $r$ with $d | 2n$ with the appropriate
polynomial evaluations.  
Our proof breaks up into several cases depending on the parity of $d$
and whether $d$ divides $n$.

If $d$ is even and $d | n$, we have that $h_{k - 2 \ell}(1,(\zeta^d)^n) = 
h_{k - 2 \ell}(1^2) = k - 2 \ell + 1$ for all $\ell$.  Pieri's Rule, 
Corollary 4.3, and Lemma 4.5 imply that the polynomial evaluation
$X(\zeta^d)$ is equal to the number of D-multidissections of
$\mathbb{P}_{2d}$ with $\frac{kd}{2n}$ edges.  By Part 1 of Lemma 4.4,
this polynomial evaluation is the fixed point set cardinality 
$|X^{r^d}|$.  The case of $d$ even with $d$ not dividing $n$ is similar, 
but uses Part 2 of Lemma 4.4.

If $d$ is odd, for 
$\ell \leq \lfloor \frac{k}{2} \rfloor$ we have that $h_{k - 2 \ell}(1,(\zeta^d)^n) = 
h_{k - 2 \ell}(1,-1)$, which is $0$ if $k$ is odd and $1$ if $k$ is even. 
An easy exercise using 
Pieri's Rule implies that
$X(\zeta^d)$ is equal to
$h_{(\frac{k}{2},\frac{k}{2})}(1,\zeta^{2d}, \dots, \zeta^{2(n-1)d})$ if 
$k$ is even and $X(\zeta^d) = 0$ if $k$ is odd.  On the other hand, since
$r^d$ reverses diameter colors, no $r^d$-invariant D-multidissection of
$\mathbb{P}_{2n}$ can contain a diameter.  It follows that 
$r^d$-invariant D-multidissections of $\mathbb{P}_{2n}$ are in natural
bijection with C-multidissections of $\mathbb{P}_{2n}$ which are 
invariant under $d$ powers of rotation.  The equality of the fixed point
set cardinality $|X^{r^d}|$ and the polynomial evaluation
$X(\zeta^d)$ is a consequence of Theorem 3.4. 
\end{proof}

Our proof of Theorem 4.6 relied both on the algebraic result in Lemma 4.2
and the combinatorial result in Lemma 4.4.  A purely algebraic approach
is possible modulo the following conjecture regarding a possible basis
for the space $V^D(n,k)$.  For $n \geq 2$, 
recall that we
label the vertices of 
$\mathbb{P}_{2n}$ clockwise with $1, 2, \dots, n, \bar{1}, \bar{2}, \dots,
\bar{n}$.  To any D-edge $e$ in $\mathbb{P}_{2n}$, we associate an element
$z_e^D \in V^D(n)$ as follows.  Abusing notation, identify polynomials
in $\mathcal{A}_{n+2}$ with their images in the quotient $V^D(n)$.  If 
$e$ is a solid/blue diameter of the form $i \bar{i}$ for $1 \leq i \leq n$, let
$z_e^D = \Delta_{i,n+1}$.  If $e$ is a dotted/red diameter of the form
$i \bar{i}$ for $1 \leq i \leq n$, let $z_e^D = \Delta_{i,n+2}$.  If
$e$ is a pair of centrally symmetric nondiameters of the form
$ij, \bar{i}\bar{j}$ for $1 \leq i < j \leq n$, let 
$z_e^D = \Delta_{i,n+1} \Delta_{j,n+2} + \Delta_{ij}$.  Finally, if $e$
is a pair of centrally symmetric nondiameters of the form
$i \bar{j}, \bar{i} j$ for $1 \leq i < j \leq n$, let 
$z_e^D = \Delta_{i,n+1} \Delta_{j,n+2} - \Delta_{ij}$.  Observe that
$z_e^D \in V^D(n,1)$ if $e$ is a diameter of either color and 
$z_e^D \in V^D(n,2)$ if $e$ is a pair of centrally symmetric
nondiameters.  Given a D-multidissection $f$, define 
$z_f^D \in V^D(n)$ by
\begin{equation*}
z_f^D := \prod_{e \in E_D}(z_e^D)^{f(e)}.
\end{equation*}
For example, if $f$ is the D-multidissection of $\mathbb{P}_8$ on the right
of Figure 4.1, then
\begin{equation*}
z_f^D = (\Delta_{25}\Delta_{46}+\Delta_{24})(\Delta_{15}\Delta_{46}-
\Delta_{14})^2(\Delta_{25})(\Delta_{26})^2 \in V^D(4,9).
\end{equation*}
\begin{conj}
The set $\{ z_f^D \}$, where $f$ ranges over all D-multidissections
of $\mathbb{P}_{2n}$, is a $\mathbb{C}$-basis for $V^D(n)$.
\end{conj}
The geometric realization of the type D cluster algebras in 
\cite{FZClusterII} implies that one need only show that the 
set of Conjecture 4.7 spans $V^D(n)$ or is linearly independent.  
The polynomials which we have attached to D-edges do not satisfy 
the type D exchange relations and are not related to the cluster monomial
basis presented in \cite{FZClusterII} via a unitriangular
transition matrix.  Assuming Conjecture 4.7 is true, we can give
the following alternative proof of Theorem 4.6.
\begin{proof} (of Theorem 4.6, assuming Conjecture 4.7)
Embed the direct product $\mathfrak{S}_n \times \mathfrak{S}_2$ of 
symmetric groups into $GL_n(\mathbb{C}) \times GL_2(\mathbb{C})$ via
permutation matrices.  
Writing permutations in cycle notation,
let $g^D$ be the image of
$(1,2,\dots,n) \times (1,2)$ under this embedding.  It is routine
to verify that $g^D$ maps the module element $z_e^D$ to the module
element $z_{r.e}^D$ for all D-edges $e$, where $r$ is the D-rotation 
operator.  The homogeneity of the $z_e^D$ combined with Conjecture 4.7
implies that the set $\{ z_f^D \}$, where $f$ ranges over all
D-multidissections of $\mathbb{P}_{2n}$ with $k$ edges, forms 
a $\mathbb{C}$-basis for the space $V^D(n,k)$.  The desired CSP follows
from using the Weyl character evaluation in Lemma 4.2 to calculate
the traces of powers of the operator $g^D$. 
\end{proof}

We close by noting that the acting groups in our representation theoretic
proofs in types A and C and our conjectural representation theoretic proof
in type D are not equal to the associated Lie group outside of type A.  In
the type C case, the action 
involved
can be reformulated as an action of the intersection
symplectic group $Sp_{2n}(\mathbb{C}) = \{ A \in GL_{2n}(\mathbb{C}) \,|\,
A \begin{pmatrix} 0 & I_n \\ -I_n & 0 \end{pmatrix} A^T =
\begin{pmatrix} 0 & I_n \\ -I_n & 0 \end{pmatrix} \}$ with the Levi
subgroup
$GL_n(\mathbb{C}) \times GL_n(\mathbb{C})$ of block diagonal matrices in
$GL_{2n}(\mathbb{C})$.  In type D, we can see no obvious relation between
the acting group $GL_n(\mathbb{C}) \times GL_2(\mathbb{C})$ and the even
special orthogonal groups.  This is perhaps not surprising given
that the geometric realizations of finite type cluster algebras typically
have little obvious connection with the Lie group of the same 
Cartan-Killing type.

\section{Acknowledgments}
The author is grateful to Sergey Fomin, Gregg Musiker, and 
Vic Reiner for helpful conversations.  The author would also like to thank
an anonymous referee for streamlining the proofs of Lemmas 3.1 and 4.2. 

\bibliography{../bib/my}

\begin{thebibliography}{10}
\expandafter\ifx\csname url\endcsname\relax
  \def\url#1{\texttt{#1}}\fi
\expandafter\ifx\csname urlprefix\endcsname\relax\def\urlprefix{URL }\fi

\bibitem{Cayley}
A.~Cayley, On the partitions of a polygon, Proc. London Math. Soc. (1) 22
  (1890-1891) 237--262.

\bibitem{EuFu}
S.-P. Eu, T.-S. Fu, The cyclic sieving phenomenon for faces of generalized
  cluster complexes, Adv. Appl. Math. 40~(3) (2008) 350--376.

\bibitem{FR}
S.~Fomin, N.~Reading, Generalized cluster complexes and {C}oxeter
  combinatorics, Int. Math. Res. Notices 44 (2005) 2709--2757.

\bibitem{FZClusterI}
S.~Fomin, A.~Zelevinsky, Cluster algebras {I:} {F}oundations, J. Amer. Math.
  Soc. 15 (2002) 497--529.

\bibitem{FZClusterII}
S.~Fomin, A.~Zelevinsky, Cluster algebras {II:} finite type classification,
  Invent. Math. 154 (2003) 63--121.

\bibitem{FZY}
S.~Fomin, A.~Zelevinsky, Y-systems and generalized associahedra, Ann. Math. 158
  (2003) 977--1018.

\bibitem{FRTHook}
J.~S. Frame, G.~B. Robinson, R.~M. Thrall, The hook graphs of the symmetric
  group, Canad. J. Math. 6 (1954) 316--325.

\bibitem{KungRota}
J.~P.~S. Kung, G.-C. Rota, The invariant theory of binary forms, Bull. Amer.
  Math. Soc. (N. S.) 10~(1) (1984) 27--85.

\bibitem{PN}
P.~Pylyavskyy, {Non-crossing} tableaux, Ann. Comb. 13~(3) (2009) 323--339.

\bibitem{RSWCSP}
V.~Reiner, D.~Stanton, D.~White, The cyclic sieving phenomenon, J. Combin.
  Theory Ser. A 108 (2004) 17--50.

\bibitem{RCSP}
B.~Rhoades, Cyclic sieving, promotion, and representation theory, J. Combin.
  Theory Ser. A 117~(1) (2010) 38--76.

\bibitem{Sag}
B.~Sagan, The Symmetric Group, Springer, New York, 2001.

\bibitem{StanEC2}
R.~Stanley, {Enumerative Combinatorics}, vol.~2, Cambridge University Press,
  Cambridge, 1999.

\bibitem{StemTab}
J.~Stembridge, Canonical bases and self-evacuating tableaux, Duke Math. J.
  82~(3) (1996) 585--606.

\end{thebibliography}
\end{document}